\documentclass[11pt]{amsart}

\usepackage{amsmath,amscd,amssymb,latexsym}

\newcommand{\sm}{\left(\smallmatrix}
\newcommand{\esm}{\endsmallmatrix\right)}
\newcommand{\mat}{\begin{pmatrix}}
\newcommand{\emat}{\end{pmatrix}}

\renewcommand{\a}{\alpha}
\renewcommand{\b}{\beta}
\renewcommand{\i}{\infty}
\newcommand{\G}{\Gamma}
\newcommand{\g}{\gamma}
\newcommand{\ve}{\varepsilon}
\newcommand{\vp}{\varphi}
\renewcommand{\i}{\infty}
\newcommand{\lt}{\left}
\newcommand{\rt}{\right}

\newcommand{\Z}{\mathbb Z}
\newcommand{\C}{\mathbb C}

\renewcommand{\H}{\mathbb H}

\newcommand{\M}{\mathcal M}
\newcommand{\W}{\mathcal W}
\newcommand{\pds}{\frac{\partial}{\partial s}}
\newtheorem{thm}{Theorem}
\newtheorem{lem}[thm]{Lemma}
\newtheorem{cor}[thm]{Corollary}
\newtheorem{prop}[thm]{Proposition}

\theoremstyle{remark}

\numberwithin{equation}{section}
\numberwithin{thm}{section}
\begin{document}

\title[Weak Maass-Poincar\'{e} Series]{Weak Maass-Poincar\'{e} series\\ and weight $3/2$ mock modular forms}

\author{Daeyeol Jeon, Soon-Yi Kang and Chang Heon Kim}
\address{Department of Mathematics Education, Kongju National University, Kongju, 314-701, Korea}
\email{dyjeon@kongju.ac.kr}
\address{Department of Mathematics, Kangwon National University, Chuncheon, 200-701, Korea} \email{sy2kang@kangwon.ac.kr}
\address{Department of Mathematics, Hanyang University,
 Seoul, 133-791 Korea}
\email{chhkim@hanyang.ac.kr}

 \begin{abstract} The primary goal of this paper is to construct the basis of the space of weight $3/2$ mock modular forms, which is an extension of the Borcherds-Zagier basis of weight $3/2$ weakly holomorphic modular forms. The shadows of the members of this basis form the Borcherds-Zagier basis of the space of weight $1/2$ weakly holomorphic modular forms. In the course of the construction, we provide a full computation of the Fourier coefficients for the weak Maass-Poincar\'{e} Series in most general form for the purpose of future reference.
\end{abstract}
\maketitle

\renewcommand{\thefootnote}%
             {}
\section{Introduction}\label{intro}

  A {\it{weak Maass form}} of weight $k$ for the group  $\G_0(N)$ is a smooth function defined in the upper half plane which is invariant under the action of all elements in $\G_0(N)$ and is an eigenform of the weight $k$ hyperbolic Laplace operator. In addition, a weak Maass form has at most exponential growth at all cusps (See \cite[Section 3]{BF}). We call a weak Maass form {\it{harmonic}} if it is annihilated by the hyperbolic Laplacian and call a harmonic weak Maass form a {\it{weakly holomorphic modular form}} if it is holomorphic in the upper half plane with possible poles at the cusps. In fact, by its definition, the image of a harmonic weak Maass form under the differential operator $\displaystyle{\xi_k:=2iy^{k}\frac{\overline{\partial }}{\partial \bar{z}}}$ is always a weakly holomorphic modular form of weight $2-k$.
For $k\in \frac12 \Z$, let $M^!_k(N)$ denote the infinite dimensional vector space of weight $k$ weakly holomorphic modular forms on $\G_0(N)$ and $M^!_k$ denote the subspace of $M^!_k(4)$, in which each form satisfies Kohnen's plus space condition, that is, its Fourier expansion has the form $\sum a(n)q^n $ where $a(n)$ is non-zero only for integers $n$ satisfying $(-1)^{k-1/2}n\equiv 0,1 \pmod 4$. There are natural infinite bases $\{f_d|d\equiv 0,1\pmod 4,\  d\leq 0\}$ and $\{g_D|D\equiv 0,3\pmod 4,\ D<0\}$ for $M^!_\frac12$ and $M^!_\frac32$, respectively, where
$f_d$ and $g_D$ have Fourier developments of the forms
\begin{equation}\label{fn}f_d(z)=q^{d}+\sum_{D<0} A(D,d)q^{|D|} \in M^!_\frac12\end{equation} and
\begin{equation}\label{gm}g_D(z)=q^{D}+\sum_{d\leq 0} B(D,d)q^{|d|}\in M^!_\frac32\end{equation}
where $q=\exp(2\pi iz)$ and $z=x+iy\in\H$, the upper half plane.  Surprisingly, $A(D,d)=-B(D,d)$ as proved by Zagier in \cite{Zagier}.

Recently, Duke, Imamo\={g}lu and  T\'{o}th \cite{DIT} extended the basis $\{f_d|d\equiv 0,1\pmod 4,\  d\leq 0\}$ for $M^!_\frac12$ to a basis $\{h_d|d\equiv 0,1 \pmod 4 \}$ for  $H^!_\frac12$, the space of weight $1/2$ \textit{weakly harmonic Maass forms} on $\G_0(4)$, which are harmonic weak Maass forms satisfying the plus space condition. They showed that $h_d=f_d$ for $d\leq 0$ and  $\xi_{1/2}h_d=-2\sqrt{d}g_{-d}$ for $d>0$.

In fact, every harmonic weak Maass form $h(z)$ is decomposed to $h(z)=h^+(z)+h^-(z)$, where $h^+(z)$ is holomorphic and $h^-(z)$ is non-holomorphic on $\H$. If $h$ has weight $k$, following Zagier \cite{Z1}, we call the holomorphic part $h^+(z)$ \textit{a mock modular form of weight $k$} and the image of the non-holomorphic part $h^-$ under $\xi_k$
\textit{the shadow} of the mock modular form $h^+$. Note that $\xi_k(h)=\xi_k(h^-)$. Hence in \cite{DIT}, Duke, Imamo\={g}lu and  T\'{o}th eventually constructed the weight $1/2$ mock modular form $f_d$ with shadow $g_{-d}$ for each $0<d\equiv 0,1 \pmod 4$.  Furthermore, they found quite a few of interesting properties of the coefficients of the mock modular form $f_d$ such as a relation with cycle integrals of the $j$-function and with modular integrals having rational period functions.

A natural question at this point is whether there exists a mock modular form $g_D$ having $f_{-D}$ as its shadow for each $0\leq D\equiv 0,3 \pmod 4$ and the duality relation between $f_d$ ($d\leq 0$) and $g_D$ ($D<0$) is preserved for $f_d$ ($d>0$) and $g_D$ ($D\geq 0$), if there is any $g_D$ ($D\geq 0$). The existence of such mock modular forms is proved in \cite[Theorem 3.7]{BF} and \cite[Theorem 1.1]{Bru} without an explicit construction. The primary goal of this paper is thus to construct the desired mock modular forms $g_D$.

\begin{thm}\label{mainmock}  There exists a unique basis $\{g_{D}|D\equiv 0,3 \pmod 4,\ D\geq 0 \}$ for the space of mock modular forms of weight $3/2$ satisfying that each $g_D$ has shadow $f_{-D}$ and has a Fourier expansion of the form
\begin{equation}\label{gmm}g_D(z)=\sum_{n\geq 0} b(D,n)q^{n}.\end{equation}

\end{thm}

The recent advances in the theory of weak Maass forms have come along with rapid development in areas of mock modular forms, traces of CM values, and other related subjects. In many of such work, weak Maass-Poincar\'{e} Series played a key role (for example, \cite{BO1,BO2,Br,BJO,DIT,KK,MP}). While we will discuss arithmetic properties of the coefficients of $g_D$ ($D\geq 0$) in the subsequent paper \cite{JKK-sesqui}, in the course of proving Theorem \ref{mainmock} we provide a full computation of the Fourier coefficients for the weak Maass-Poincar\'{e} Series in most general form for the purpose of future reference.

We let $H^!_k(N)$ denote the space of harmonic weak Maass forms for $\G_0(N)$ and $H_k(N)$ denote its subspace, in which the image of each form under $\xi_k$ is a cusp form. As usual, $M_k(N)$ and $S_k(N)$ denote the spaces of classical modular forms and cusp forms for $\G_0(N)$, respectively.

\section*{Acknowledgement}

The authors give many thanks \"{O}zlem Imamo\={g}lu for her constructive suggestion.

\section{Whittaker functions}
We first recall many properties of Whittaker functions, with which we construct weak Maass-Poincar\'{e} Series as in  \cite{BO1,BO2,Br,BJO,DIT,KK,MP}. Whittaker functions $M_{\mu,\nu}(y)$ and $W_{\mu,\nu}(y)$ are as defined in \cite[ch.13]{AS}, linearly independent solutions of the Whittaker differential equation
\begin{equation}\label{de}
\frac{d^2w}{dy^2}+\lt(-\frac 14+\frac{\mu}{y}+\frac{\frac 14-\nu^2}{y^{2}}\rt)w=0.\end{equation}
 These functions can be expressed in terms of confluent hypergeometric functions as
\begin{eqnarray}\label{Kummer}
&M_{\mu,\nu}(y)=e^{-y/2}y^{\nu+1/2}M(\nu-\mu+\frac12,1+2\nu,y)\\
&W_{\mu,\nu}(y)=e^{-y/2}y^{\nu+1/2}U(\nu-\mu+\frac12,1+2\nu,y),
\end{eqnarray}
where
\begin{equation}\label{hypM}
M(a,b,x)=\sum_{n=0}^\i\frac{(a)(a+1)(a+2)\cdots(a+n-1)}{(b)(b+1)(b+2)\cdots(b+n-1)}\frac{x^n}{n!}
\end{equation}
and
\begin{equation}\label{hypU}
U(a,b,x)=\frac{\G(1-b)}{\G(a-b+1)}M(a,b,x)+\frac{\G(b-1)}{\G(a)}x^{1-b}M(a-b+1,2-b,x).
\end{equation}
Here $\G(z)$ is the Gamma function. The equation (\ref{hypU}) implies, if $2\nu \notin \mathbb Z$,
\begin{equation}\label{mwrel}
W_{\mu,\nu}(y)=\frac{\G(-2\nu)}{\G(\frac 12-\nu-\mu)}M_{\mu,\nu}(y)+\frac{\G(2\nu)}{\G(\frac 12+\nu-\mu)}M_{\mu,-\nu}(y)\end{equation}
 and in particular, $W_{\mu,\nu}(y)=W_{\mu,-\nu}(y)$.

The Whittaker functions may have integral representations for certain fixed values of $\mu$ and $\nu$. If $\mathrm{Re}(\nu\pm \mu+1/2)>0$, then
\begin{equation}\label{im}
M_{\mu,\nu}(y)=y^{\nu+1/2}e^{y/2}\frac{\G(1+2\nu)}{\G(\nu+\mu+1/2)\G(\nu-\mu+1/2)}\int_0^1t^{\nu+\mu-1/2}(1-t)^{\nu-\mu-1/2}e^{-yt}dt\end{equation}
and if $\mathrm{Re}(\nu-\mu+1/2)>0$, then
\begin{equation}\label{iw}
W_{\mu,\nu}(y)=y^{\nu+1/2}e^{y/2}\frac{1}{\G(\nu-\mu+1/2)}\int_1^\i t^{\nu+\mu-1/2}(t-1)^{\nu-\mu-1/2}e^{-yt}dt.\end{equation}
It follows from (\ref{im}) and (\ref{iw}) that if $\nu-\mu=1/2$, then we have
\begin{equation}\label{wmw}
M_{\mu,\nu}(y)+(2\mu+1)W_{\mu,\nu}(y)=\G(2\mu+2)y^{-\mu}e^{y/2}, \end{equation} and from (\ref{iw}) that if $\nu+\mu=1/2$, we have
\begin{equation}\label{wws}W_{\mu,\nu}(y)=y^\mu e^{-y/2}.\end{equation}
If $\frac12-\mu\pm \nu$ is an integer, Whittaker functions can be expressed as an incomplete gamma function $\G(a,x)=\int_x^\i e^{-t}t^{a-1} dt$. For example, if $\nu\in \frac 12\Z$, then
\begin{equation}\label{wi}
W_{\nu-1/2,\nu}(y)=e^{y/2}y^{1/2-\nu}\G(2\nu,y).\end{equation}
Asymptotic behavior of the Whittaker functions for fixed $\mu$, $\nu$ is also found from (\ref{im}) and (\ref{iw}):
\begin{equation}\label{wasy}
M_{\mu,\nu}(y)\sim \frac{\G(1+2\nu)}{\G(\nu-\mu+1/2)}y^{-\mu}e^{y/2} \quad \mathrm{and}\quad
W_{\mu,\nu}(y)\sim y^{\mu}e^{-y/2}\quad \mathrm{as}\quad y\to \i,\end{equation}
\begin{equation}\label{masy}
M_{\mu,\nu}(y)\sim y^{\nu+1/2} \quad \mathrm{and}\quad
W_{\mu,\nu}(y)\sim \frac{\G(2\nu)}{\G(\nu-\mu+1/2)}y^{-\nu+1/2}\quad \mathrm{as}\quad y\to 0\end{equation}

Now we define for fixed values $s\in \mathbb{C}$ and $n\in \mathbb{Z}$,
\begin{eqnarray}\label{defmnwn}
&&\mathcal{M}_{n,k}(y,s)=\left\{
                     \begin{array}{ll}
                       \G(2s)^{-1}(4\pi |n|y)^{-k/2}M_{\frac k2\mathrm{sgn}(n),s-1/2}(4\pi |n|y), & \hbox{if $n\neq 0$,} \\
                       y^{s-k/2}, & \hbox{if $n=0$.}
                     \end{array}
                   \right.\\ \nonumber
&&\mathcal{W}_{n,k}(y,s)=\left\{
  \begin{array}{ll}
   \G(s+\frac k2\mathrm{sgn}(n))^{-1}|n|^{k/2-1}(4\pi y)^{-k/2}W_{\frac k2\mathrm{sgn}(n),s-1/2}(4\pi |n| y), &  \hbox{if $n\neq 0$,}\\
   \frac{(4\pi)^{1-k}y^{1-s-k/2}}{(2s-1)\G(s-k/2)\G(s+k/2)}, &  \hbox{if $n=0$.}
  \end{array}
\right.\end{eqnarray}
These are generalizations of the functions treated in \cite[Chapter 1.3]{Br}, \cite[Section 3]{BJO} and \cite[Section 2]{DIT} preserving  important properties of them. In particular, if $e(z)=\exp(2\pi i z)$, then the function
\begin{equation}\label{defp}\vp_{m,k}(z,s):=\mathcal{M}_{m,k}(y,s)e(mx)\end{equation} is an eigenfunction of the weight $k$ hyperbolic Laplacian
$$\Delta_k=-y^2\lt(\frac{\partial^2}{\partial x^2}+\frac{\partial^2}{\partial y^2}\rt)+iky\lt(\frac{\partial}{\partial x}+i\frac{\partial}{\partial y}\rt)$$
and has eigenvalue $s(1-s)+(k^2-2k)/4$ by (\ref{de}). That is,
\begin{equation}\label{hypm}
\Delta_k\vp_{m,k}(z,s)=\lt(s-\frac{k}{2}\rt)\lt(1-\frac k2-s\rt)\vp_{m,k}(z,s).\end{equation}
Also, due to the asymptotic behavior of the Whittaker function given in (\ref{masy}),
\begin{equation}\label{asymphi}\vp_{m,k}(z,s)=O(y^{\mathrm{Re}(s)-k/2})\quad \mathrm{as}\quad y\to 0.\end{equation}

We are interested in the values of $s=k/2$ and $s=1-k/2$, for which $\Delta_k\vp_{m,k}(z,s)=0$. It follows from (\ref{defmnwn}) and (\ref{wws}) for the case $n>0$, from  (\ref{defmnwn}) and (\ref{iw}) for the case $n<0$ and from  (\ref{defmnwn}) for the case $n=0$ that
\begin{equation}\label{w1-k2}
\mathcal{W}_{n,k}(y,1-k/2)=e^{-2\pi ny}\left\{
  \begin{array}{ll}
   n^{k-1}, &  \hbox{if $n> 0$,}\\
  |n|^{k-1} \G(1-k)^{-1}\G(1-k,-4\pi n y), &  \hbox{if $n<0$,}\\
  \frac{(4\pi)^{1-k}}{\G(2-k)} , &  \hbox{if $n=0$.}
  \end{array}
\right.
\end{equation}
If we assume $k\leq 1/2$, then by using  (\ref{defmnwn}) and (\ref{im}), we find that
\begin{equation}\label{m1-k2}
 \mathcal{M}_{n,k}(y,1-k/2)=e^{-2\pi ny} \left\{
  \begin{array}{ll}
  (-1)^{k}\lt[\G(1-k)^{-1}\G(1-k,-4\pi n y)-1\rt], & \hbox{ if $n>0$,}\\
  1-\G(1-k)^{-1}\G(1-k,-4\pi n y), & \hbox{ if $n<0 $,}\\
     y^{1-k}, & \hbox{ if $n=0$.}\\
  \end{array}
\right.\end{equation}
Also, using the fact $W_{\mu,\nu}(y)=W_{\mu,-\nu}(y)$ and (\ref{wws}) for the case $n>0$ and directly from (\ref{defmnwn}) for other cases, we have
\begin{equation}\label{wk2}
\mathcal{W}_{n,k}(y,k/2)=e^{-2\pi ny}\left\{
  \begin{array}{ll}
   \G(k)^{-1}n^{k-1}, &  \hbox{if $n> 0$,}\\
    0, &  \hbox{if $n\leq 0$.}
     \end{array}
\right.
\end{equation}
Finally, using (\ref{defmnwn}) and (\ref{Kummer}), we obtain
\begin{equation}\label{mk2}
\mathcal{M}_{n,k}(y,k/2)=e^{-2\pi ny}\left\{
  \begin{array}{ll}
   \G(k)^{-1}, &  \hbox{if $n\neq 0$,}\\
       1 , &  \hbox{if $n=0$.}
  \end{array}
\right.
\end{equation}

\section{Weak Maass-Poincar\'{e} series}
Consider the group $$\mathfrak{G}:=\{(A,\phi(z))| A=\sm a & b \\ c & d \esm\in GL^+_2(\mathbb{R}),
\ \phi:\mathbb{H}\to \mathbb{C}\ \mathrm{hol},\ |\phi(z)|=(\det A)^{-1/4}|cz+d|^{1/2}\}$$
with group law $(A_1,\phi_1(z))(A_2,\phi_2(z))=(A_1A_2,\phi_1(A_2z)\phi_2(z)).$ $\mathfrak{G}$ acts on complex valued functions $f$ defined on the upper half plane by $(f|_k(A,\phi)(z))=\phi(z)^{-2k}f(Az).$
For $\g=\sm a & b \\ c & d\esm \in \Gamma_0(N)$, define $j(\g,z)$ by
\begin{equation*}\label{j}j(\g,z)=\left\{
  \begin{array}{ll}
    \sqrt{cz+d}, & \hbox{if $k\in\mathbb{Z}$,} \\
    \lt(\frac cd\rt)\ve_d^{-1}\sqrt{cz+d}, & \hbox{if $k\in\frac12 \mathbb{Z}\backslash \mathbb{Z}$,}
  \end{array}
\right.\end{equation*}
where $\sqrt z$ is the principal branch of the holomorphic square root and
\begin{equation*}\label{ve}\ve_d:=\left\{
                     \begin{array}{ll}
                       1, & \hbox{ if $d\equiv 1 \pmod 4$,} \\
                       i, & \hbox{ if $d\equiv 3 \pmod 4$.}
                     \end{array}
                   \right.\end{equation*}
The map $\G_0(N)\to \mathfrak{G}$ defined by $\g\mapsto \widetilde{\g}=(\g,j(\g,z))$ is a group homomorphism. For convenience,  we write $f|_k\g$ instead of $f|_k\widetilde{\g}$. Namely, the weight $k$ slash operator is given by
\begin{equation*}\label{slash}
(f|_k{\g})(z)=j(\g,z)^{-2k}f(\g z).
\end{equation*}

If $\vp:=\vp_{m,k}(z,s)$ is the function defined in (\ref{defp}) and  $\G_\i=\{\pm\sm 1&n\\0&1\esm|n\in \mathbb{Z}\}$ is the subgroup of translations of $\G_0(N)$, then the Poincar\'{e} series
\begin{equation}\label{poincF}F_{m,k,N}(z,s):=
\sum_{\g\in \G_\i\backslash\G_0(N)}(\vp|_k \g)(z)\end{equation}
 is $\G_0(N)$-invariant and it follows from (\ref{asymphi}) that  $F_{m,k,N}(z,s)$ converges absolutely and uniformly on compacta for $\mathrm{Re}(s)>1$. Moreover, $F_{m,k,N}(z,s)$ is an eigenfunction of the Laplacian $\Delta_k$. In proving these, we use the same arguments with those in \cite{Br} and \cite[p.~8]{DIT}.

\begin{lem}
Let $m\in \mathbb Z$ and $s\in \mathbb C$. The Poincar\'{e} series $F_{m,k,N}(z,s)$ converges absolutely and uniformly on compacta for $\mathrm{Re}(s)>1$, and it is a $\G_0(N)$-invariant eigenfunction of the Laplacian $\Delta_k$ satisfying
\begin{equation}\label{hypp}
\Delta_kF_{m,k,N}(z,s)=\lt(s-\frac{k}{2}\rt)\lt(1-\frac k2-s\rt)F_{m,k,N}(z,s).\end{equation} Thus $F_{m,k,N}(z,s)$ is real analytic. \end{lem}
\begin{proof} By \cite[Lemma 3.1]{BO2} and (\ref{asymphi}), if $k>2-2(\mathrm{Re}(s)-k/2)$, then $F_{m,k,N}(z,s)$ is convergent and $\G_0(N)$-invariant. On the other hand, the Laplace operator $\Delta_k$ can be expressed in terms of the differential operator $\xi_k$
\begin{equation}\label{Dx}\Delta_k=-\xi_{2-k}\circ \xi_k.\end{equation}
The operator $\xi_k$ and the slash operator commute under the rule  $\xi_k(f|_k g(z))=(\xi_k f)|_{2-k}(g z)$. Hence
if we write $F:=F_{m,k,N}(z,s)$ and $\Delta:=\Delta_k$, then
 \begin{eqnarray*}
 &&\Delta F=\Delta\sum(\vp|_k \g)=\sum(-\xi_{2-k}\circ\xi_k)(\vp|_k \g)=-\sum\xi_{2-k}(\xi_k(\vp|_k \g))\\
 &&\qquad =-\sum\xi_{2-k}(\xi_k(\vp)|_{2-k} \g)=\sum(\Delta\vp)|_k\g=(s(1-s)+(k^2-2k)/4)F, \end{eqnarray*}
where the last equality follows from (\ref{hypm}).  \end{proof}

The function $F_{0,1/2,4}(z,s)$ is the usual weight $1/2$ Eisenstein series and $F_{m,1/2,4}(z,s)$ for any integer $m$ is briefly discussed in \cite[Lemma 2]{DIT}. The Fourier expansion of the Poincar\'{e} series were computed earlier for negative integer $m$ and negative weight $k$ in \cite[Theorem 1.9]{Br} and for negative integer $m$ and weight $k\geq 2$ in \cite[Proposition 3.2]{BJO}. We now provide the Fourier expansion of $F_{m,k,N}(z,s)$ for any integer $m$, arbitrary weight $k$ and level $N$.

\begin{thm}\label{exp1}
If $m$ is an integer, then the Poincar\'{e} series $F_{m,k,N}(z,s)$ has the Fourier expansion
$$F_{m,k,N}(z,s)=\mathcal{M}_{m,k}(y,s)e(mx)+\sum_{n\in \mathbb{Z}} c_{m,k}(n,s)\mathcal{W}_{n,k}(y,s) e(nx),$$
where the coefficients $c_{m,k}(n,s)$ are given by
$$(2\pi i^{-k})\sum_{c>0} \frac{K_{k}(m,n,Nc)}{Nc}\times \left\{
    \begin{array}{ll}
      \displaystyle{|mn|^{\frac{1-k}{2}}J_{2s-1}\lt(\frac{4\pi\sqrt{|mn|}}{Nc}\rt)}, & mn>0, \\
      \displaystyle{|mn|^{\frac{1-k}{2}}I_{2s-1}\lt(\frac{4\pi\sqrt{|mn|}}{Nc}\rt)}, & mn<0, \\
    \displaystyle{2^{k-1}\pi^{s+\frac{k}{2}-1}|m+n|^{s-\frac{k}{2}}(Nc)^{1-2s}}, & mn=0, m+n\neq 0,
\\
    \displaystyle{2^{2k-2}\pi^{k-1}\Gamma(2s)(2Nc)^{1-2s}}, & m=n=0,
    \end{array}
  \right.
$$
where $J_\a$ and $I_\a$ are the usual Bessel functions as defined in \cite[Ch.~9]{AS} and $K_k(m,n,c)$ is the (generalized) Kloosterman sum defined by
\begin{equation}\label{kloos}
K_k(m,n,c):=\left\{
              \begin{array}{ll}
                \sum_{v (c)^*}e\left(\frac{m\bar{v}+nv}{c}\right), & \hbox{if $k\in\mathbb{Z}$,} \\
                \sum_{v (c)^*}\lt(\frac cv\rt)^{2k}\ve_v^{2k}e\left(\frac{m\bar{v}+nv}{c}\right), & \hbox{if $k\in\frac12 \mathbb{Z}\backslash \mathbb{Z}$.}
              \end{array}
            \right.
\end{equation}
Here the sum runs through the primitive residue classes modulo $c$ and $v\bar{v}\equiv 1 \pmod c$.
\end{thm}

\begin{proof} By the exactly same argument in the proof of \cite[Proposition 4]{Kohnen}, we have
$$F_{m,k,N}(z,s)=\mathcal M_{m,k}(y,s)e(mx)+\sum_{c>0,d(Nc)^*}C(c,d,N)\sum_{r\in\Z}\lt(Nc(z+r)+d\rt)^{-k}
\varphi_{m,k}\lt(\frac{a(z+r)+b}{Nc(z+r)+d},s\rt),$$ where
$C(c,d,N)=1$ if $k$ is an integer and $C(c,d,N)=\lt(\frac{Nc}{d}\rt)\lt(\frac{-4}{d}\rt)^k$ if $k$ is not an integer but a half integer.
Now the theorem follows from Lemma \ref{gendit} below.
\end{proof}

 It follows from (\ref{hypp}) that $\Delta_kF_{m,k,N}(z,s)=0$ for the special $s$ values $s=k/2$ and $s=1-k/2$. If $k\leq 1/2$, then $F_{m,k,N}(z,s)$ is holomorphic in $s$ near $s=1-k/2$ and if  $k\geq 3/2$, then $F_{m,k,N}(z,s)$ is holomorphic in $s$ near $s=k/2$. In fact, $F_{m,k,N}(z,k/2)$ is a weakly holomorphic modular form if $k> 2$ and $F_{m,k,N}(z,1-k/2)$ is a weak Maass form if $k<0$.

\begin{cor}\label{wtlarger2}
If $k> 2$ and $m$ is an integer, then the Fourier expansion of the Poincar\'{e} series $F_{m,k,N}(z,k/2)$ is given by
\begin{equation*}
F_{m,k,N}(z,k/2)=
\frac{\delta_m}{\G(k)}q^m+\sum_{0<n\in \mathbb{Z}} \frac{c_{m,k}(n,k/2)}{\G(k)}n^{k-1}q^n\in\left\{
              \begin{array}{ll}
 S_k(N),&\ \mathrm{if}\ m>0,\\
 M^!_k(N),&\ \mathrm{if}\ m<0,\\
 M_k(N),&\ \mathrm{if}\ m=0.\end{array}
            \right.
\end{equation*}
where $\delta_m=\G(k)$ if $m=0$ and $\delta_m=1$ otherwise.  The coefficients $c_{m,k}(n,s)$ are defined in Theorem \ref{exp1}.
\end{cor}
\begin{proof} This is an immediate consequence of Theorem \ref{exp1}, (\ref{wk2}) and (\ref{mk2}). \end{proof}
The Fourier expansion above for negative $m$, $2<k\in\mathbb Z\backslash \frac12\mathbb Z$ and $N=4$ was computed in \cite[Theorem 3.3]{BJO} with slightly different normalization. For negative weight $k$, Bruinier established the Fourier expansion earlier for negative $m$ in \cite[Proposition 1.10]{Br}. The following is its generalization into the case for every integer $m$ and any level $N$.
\begin{cor}\label{wtneg}
If $k<0$ and $m$ is an integer, then the Poincar\'{e} series $F_{m,k,N}(z,1-k/2)\in H_k(N)$ if $m\leq 0$  and $F_{m,k,N}(z,1-k/2)\in H^!_k(N)$ if $m>0$. Its Fourier expansion is given by
\begin{eqnarray*}
&&F_{m,k,N}(z,1-k/2)=\mathfrak m_{m,k}(y)q^m+\frac{(4\pi)^{1-k}}{(1-k)\G(1-k)}c_{m,k}(0,1-k/2)\\
&&\qquad\qquad+\sum_{n>0} c_{m,k}(n,1-k/2)n^{k-1}q^n
+\sum_{n<0} c_{m,k}(n,1-k/2)|n|^{k-1}\frac{\G(1-k,-4\pi ny)}{\G(1-k)}q^n,
            \end{eqnarray*}
where $$\mathfrak m_{m,k}(y):=\left\{
                     \begin{array}{ll}
                       y^{1-k} & \hbox{if $m=0$,} \\
                       1-\frac{\G(1-k,-4\pi my)}{\G(1-k)}, & \hbox{if $m<0 $,} \\
                       (-1)^{k-1}\lt[1-\frac{\G(1-k,-4\pi my)}{\G(1-k)}\rt], & \hbox{if $m>0$}
                     \end{array}
                   \right.$$
                   and the coefficients $c_{m,k}(n,s)$ are defined in Theorem \ref{exp1}.
\end{cor}
\begin{proof} This result follows from Theorem \ref{exp1}, (\ref{w1-k2}) and (\ref{m1-k2}). \end{proof}

Many important properties of the weak Maass Poincar\'{e} series when $k<0$ and $k>2$ are discussed in \cite{BKR}. For weights $0\leq k\leq 2$, the Poincar\'{e} series does not converge. The theory of the resolvent kernel, however, assures that the Poincar\'{e} series $F_{m,k,N}(z,s)$ has an analytic continuation in $s$ to  $\mathrm{Re}(s)>1/2$ except for possibly finitely many simple poles in $(1/2,1)$.  These poles may only occur at points of the discrete spectrum of $\Delta_k$. (Refer to \cite[p.~386]{BJO}, \cite[p.~12]{DIT}, \cite[Section 3]{F}, \cite{H}.) Hence Theorem \ref{exp1} gives Fourier expansions for the Maass Poincar\'{e} series of these weights as well. The weak Maass Poincar\'{e} series $F_{0,2,N}(z,s)$ and $F_{0,0,N}(z,s)$ are both usual Eisenstein series while $F_{m,0,1}(z,s)$ is the Niebur-Poincar\'{e} series which is a weak Maass form \cite{DIT, F, Neun, Niebur}. Furthermore, the harmonic weak Maass Poincar\'{e} series
\begin{equation}F_{m,0,1}(z,1)=\overline{j_m(z)} + 24 \sigma (m),\end{equation}
where $j_m(z)=q^{-m}+O(q)$ ($m\geq 0$) form a unique basis for $\C[j]$. The Poincar\'{e} series $F_{m,2,1}(z,1)$ is a weakly holomorphic modular form (\cite[Theorem 3.3]{BJO}) if $m<0$ but it becomes trivial if $m>0$. In \cite{DIT2}, Duke, Imamo\={g}lu, and T\'{o}th instead works on the derivative of $F_{m,2,1}(z,s)$ at $s=1$ which is a non-trivial harmonic weak Maass form. We also use a derivative of an appropriate Poincar\'{e} series of weight $3/2$. As half-integral weight forms have a better arithmetic meaning in the plus space, we first construct weak Maass Poincar\'{e} series satisfying the plus space condition for any weight $k\in \mathbb Z\backslash \frac12\mathbb Z$ and then discuss weak Maass forms of weights $1/2$ and $3/2$ in the plus space in the following sections.

Before we end this section, we prove Lemma \ref{gendit} to complete the proof of Theorem \ref{exp1}. This approach is well described in Kohnen's earlier work \cite{Kohnen} and we only generalize it to arbitrary weight $k$ and arbitrary level $N$. Although the same method is used in many papers, none of them provide a full computation.  We first state two results of Hejhal \cite[Lemma 5.5, p.357]{H} to prove Lemma \ref{gendit}.

\begin{prop}[Hejhal]\label{Hejhal1}
\begin{eqnarray*}
\int_{-\infty}^{\infty}(1-iu)^{-a}(1+iu)^{-b}e^{ixu}du
=\left\{\begin{array}{ll}
\dfrac{\pi 2^{-\frac{a+b}{2}+1}x^{\frac{a+b}{2}-1}}{\Gamma(b)}W_{\frac{-a+b}{2},\frac{a+b-1}{2}}(2x),&\hbox{$x>0$,}\\
\pi 2^{2-a-b}\dfrac{\Gamma(a+b-1)}{\Gamma(a)\Gamma(b)},&\hbox{$x=0$.}
\end{array}\right.\\
\end{eqnarray*}
\end{prop}

\begin{prop}[Hejhal]\label{Hejhal2}
$$\begin{array}{l}
\displaystyle\int_{-\infty}^{\infty}e^{2\pi ic\frac{u}{1+u^2}}e^{2\pi i\alpha u}
\left(\frac{1-iu}{1+iu}\right)^{l}
\frac{M_{l,s-\frac12}\left(\frac{4\pi c}{1+u^2}\right)}{\Gamma(2s)} du \\
\hspace{4cm}=\left\{\begin{array}{ll}
\dfrac{2\pi}{\Gamma(s+l)}W_{l,s-\frac{1}{2}}(4\pi\alpha)\frac{\sqrt{c}}{\sqrt{\alpha}}J_{2s-1}(4\pi\sqrt{\alpha c}),&\hbox{$\alpha>0$,}\\
\dfrac{2\pi^{s+1}}{(s-\frac12)\Gamma(s-l)\Gamma(s+l)}c^s,&\hbox{$\alpha=0$,}\\
\dfrac{2\pi}{\Gamma(s-l)}W_{-l,s-\frac{1}{2}}(4\pi\alpha)\frac{\sqrt{c}}{\sqrt{|\alpha|}}I_{2s-1}(4\pi\sqrt{|\alpha|c}),&\hbox{$\alpha<0$.}
\end{array}\right.
\end{array}$$
\end{prop}

Using these, we then derive a generalization of \cite[Lemma, p.~253]{Kohnen} and \cite[Lemma 2]{DIT} for any $k$ and $N$.

\begin{lem}\label{gendit}
Let $\g=\sm a & b \\ c & d\esm \in SL_2(\mathbb{R})$ have $c>0$ and $\nu \in \mathbb{N}$ and suppose that $\mathrm{Re}(s)>1/2$. Then for $\vp_{m,k}(z,s)$ defined in (\ref{defp}) with any $m\in \mathbb{Z}$, we have
\begin{eqnarray*}\sum_{r\in \mathbb{Z}}(c(z+r)+d)^{-k}\vp_{m,k}\lt(\nu\frac{a(z+r)+b}{c(z+r)+d},s\rt)
=2\pi i^{-k}\sum_{n\in \mathbb{Z}}e\lt(\frac{a\nu m+nd}{c}\rt)\mathcal{W}_{n,k}(y,s)e(nx)\\
\times\left\{
    \begin{array}{ll}
      \displaystyle{{|\nu mn|^{\frac{1-k}{2}}}c^{-1}J_{2s-1}\lt(\frac{4\pi\sqrt{|\nu mn|}}{c}\rt)}, & \hbox{if $mn>0$,} \\
      \displaystyle{{|\nu mn|^{\frac{1-k}{2}}}c^{-1}I_{2s-1}\lt(\frac{4\pi\sqrt{|\nu mn|}}{c}\rt)},& \hbox{if $mn<0$,} \\
    \displaystyle{{2^{k-1}\pi^{s+k/2-1}|\nu (m+n)|^{s-k/2}}c^{-2s}}, & \hbox{if $mn=0$, $m+n\neq 0$,} \\
    \displaystyle{{2^{2k-1}\pi^{k-1}\G(2s)}(2c)^{-2s}\nu^{s-k/2}}, & \hbox{if $m=n=0$.}
    \end{array}
  \right.
\end{eqnarray*}
\end{lem}

\begin{proof} As explained in \cite{Kohnen}, it suffices to treat the case $M=\left(\begin{array}{rr}0&-1\\1&0\end{array}\right)$. Since the series $\sum_{r\in\Z}(z+r)^{-k}\varphi_{m,k}\left(-\frac{\nu}{z+r},s\right)$ converges absolutely uniformly on compact subsets of $\mathbb H$ for $\mathrm{Re}(s)>1/2$, we may write
$$\sum_{r\in\Z}(z+r)^{-k}\varphi_{m,k}\left(-\frac{\nu}{z+r},s\right)
=\sum_{n\in\Z}\b(n,y,s)e(nx).$$
For nonzero $m\in\Z$,
\begin{eqnarray*}
\b(n,y,s)&=&\int_0^1e(-nx)\sum_{r\in\Z}(z+r)^{-k}\varphi_{m,k}\left(-\frac{\nu}{z+r},s\right)dx\\
&=&\int_{-\infty}^{\infty}e(-nx)z^{-k}\varphi_{m,k}\left(-\frac{\nu}{z},s\right)dx\\
&=&\int_{-\infty}^{\infty}z^{-k}\left(\frac{4\pi\nu |m|y}{x^2+y^2}\right)^{-\frac{k}{2}}\frac{M_{{\rm sgn}(m)\frac{k}{2},s-\frac{1}{2}}\left(\frac{4\pi\nu |m|y}{x^2+y^2}\right)}{\Gamma(2s)}e\left(-\frac{\nu mx}{x^2+y^2}-nx\right)dx\\
&=&(4\pi\nu |m|y)^{-\frac{k}{2}}i^{-k}I,\end{eqnarray*} where we set
$$I=
\int_{-\infty}^{\infty}\left(\frac{y-ix}{y+ix}\right)^{-\frac{k}{2}}\frac{M_{{\rm sgn}(m)\frac{k}{2},s-\frac{1}{2}}\left(\frac{4\pi\nu |m|y}{x^2+y^2}\right)}{\Gamma(2s)}e\left(-\frac{\nu mx}{x^2+y^2}-nx\right)dx.$$
Substituting $x=-\mathrm{sgn}(m)yu$, we have
\begin{eqnarray*}
I=y\int_{-\infty}^{\infty}\left(\frac{1-iu}{1+iu}\right)^{\mathrm{sgn}(m)\frac{k}{2}}\frac{M_{\mathrm{sgn}(m)\frac{k}{2},s-\frac{1}{2}}\left(\frac{4\pi\nu |m|}{y(u^2+1)}\right)}{\Gamma(2s)}e\left(\frac{\nu |m|u}{y(u^2+1)}+{\rm sgn}(m)nyu\right)du.
\end{eqnarray*}
Then setting $A=\mathrm{sgn}(m)ny$ and $B=\frac{\nu |m|}{y}$, we get
\begin{eqnarray*}
I=y\int_{-\infty}^{\infty}\left(\frac{1-iu}{1+iu}\right)^{\mathrm{sgn}(m)\frac{k}{2}}\frac{M_{\mathrm{sgn}(m)\frac{k}{2},s-\frac{1}{2}}\left(\frac{4\pi B}{u^2+1}\right)}{\Gamma(2s)}e\left(\frac{Bu}{u^2+1}+Au\right)du
\end{eqnarray*}
so that we may apply Proposition \ref{Hejhal2} to obtain that
\begin{eqnarray*}
I&=&y\left\{\begin{array}{ll}
\frac{2\pi}{\Gamma\left(s+{\rm sgn}(m)\frac{k}{2}\right)}W_{{\rm sgn}(m)\frac{k}{2},s-\frac{1}{2}}(4\pi A)\sqrt{\frac{B}{A}}J_{2s-1}(4\pi\sqrt{AB}),&\hbox{$A>0$,}\\
\frac{4\pi^{1+s}}{(2s-1)\Gamma\left(s+\frac{k}{2}\right)\Gamma\left(s-\frac{k}{2}\right)}B^s,&\hbox{$A=0$,}\\
\frac{2\pi}{\Gamma\left(s-{\rm sgn}(m)\frac{k}{2}\right)}W_{-{\rm sgn}(m)\frac{k}{2},s-\frac{1}{2}}(4\pi |A|)\sqrt{\frac{B}{|A|}}I_{2s-1}(4\pi\sqrt{|A|B})),&\hbox{$A<0$.}
\end{array}\right.\end{eqnarray*}
Hence we have
\begin{eqnarray*}\b(n,y,s)&=&2\pi i^{-k}\mathcal W_{n,k}(y,s)\left\{\begin{array}{ll}
{\left|\nu {n}{m}\right|^{\frac{1-k}{2}}}J_{2s-1}(4\pi\sqrt{\nu |mn|}),&\hbox{$mn>0$,}\\
2^{k-1}\pi^{s+\frac{k}{2}-1}|\nu m|^{s-\frac{k}{2}},&\hbox{$mn=0$,}\\
{\left|\nu {n}{m}\right|^{\frac{1-k}{2}}}I_{2s-1}(4\pi\sqrt{\nu |mn|}),&\hbox{$mn<0$.}
\end{array}\right.
\end{eqnarray*}
If $m=0$, then
\begin{eqnarray*}
\b(n,y,s)&=&\int_0^1e(-nx)\sum_{r\in\Z}(z+r)^{-k}\varphi_{0,k}\left(-\frac{\nu}{z+r},s\right)dx\\
&=&\int_{-\infty}^{\infty}e(-nx)z^{-k}\varphi_{0,k}\left(-\frac{\nu}{z},s\right)dx\\
&=&\int_{-\infty}^{\infty}e(-nx)z^{-k}\left(\frac{\nu y}{x^2+y^2}\right)^{s-\frac{k}{2}}dx\\
&=&(\nu y)^{s-\frac{k}{2}}\int_{-\infty}^{\infty}z^{-(s+\frac{k}{2})}\bar{z}^{-(s-\frac{k}{2})}e(-n x)dx.
\end{eqnarray*}
By letting $x=yu$ and $x=-yu$ when $n\leq 0$ and $n>0$, respectively, we have

\begin{eqnarray*}
\b(n,y,s)=\nu^{s-\frac{k}{2}}y^{-s-\frac{k}{2}+1}i^{-k}\int_{-\infty}^{\infty}(1-iu)^{-(s\pm \frac{k}{2})}
(1+iu)^{-(s\mp \frac{k}{2})}e^{2\pi i|n|yu}du.\end{eqnarray*}
If $n\neq 0$, then by Proposition \ref{Hejhal1},
\begin{eqnarray*}
\b(n,y,s)&=&\nu^{s-\frac{k}{2}}y^{-s-\frac{k}{2}+1}i^{-k}\pi 2^{-s+1}(2\pi |n|y)^{s-1}\frac{W_{\mathrm{sgn}(n)\frac k2, s-\frac 12}(4\pi |n|y)}{\Gamma(s+\mathrm{sgn}(n)\frac{k}{2})}\\
&=&2\pi i^{-k}\mathcal W_{n,k}(y,s){2^{k-1}\pi^{s+\frac{k}{2}-1}|\nu n|^{s-\frac{k}{2}}} .\end{eqnarray*}
If $n=0$, then by Proposition \ref{Hejhal1} again,
\begin{eqnarray*}
\b(n,y,s)&=&\nu^{s-\frac{k}{2}}y^{-s-\frac{k}{2}+1}i^{-k}\pi 2^{2-2s}\frac{\Gamma(2s-1)}{\Gamma\left(s+\frac{k}{2}\right)\Gamma\left(s-\frac{k}{2}\right)}\\
&=&2\pi i^{-k}\mathcal W_{0,k}(y,s) 2^{1-2s}(2s-1)\G(2s-1)(4\pi)^{k-1}\nu^{s-k/2}\\
&=&2\pi i^{-k}\mathcal W_{0,k}(y,s) 2^{-1+2k-2s}\pi^{k-1}\G(2s)\nu^{s-k/2}.
\end{eqnarray*}
\end{proof}

\section{Weak Maass-Poincar\'{e} series in the plus space}

Throughout this section, we assume $k=\lambda+\frac 12$ where $\lambda\in \mathbb{Z}$ and $N=4N'$. Following \cite[p.~250]{Kohnen}, we employ Kohnen's projection operator $pr_k^{+}$ to construct a weight $k$ Maass-Poincar\'{e} series that satisfies the plus space condition, whose Fourier coefficients are supported on $(-1)^\lambda n\equiv 0,1\pmod{4}$. The weight $k$ projection operator $pr_k^{+}$ is defined by
$$pr_k^+(f)=\frac 12 f+(-1)^{k}\frac{1}{2\sqrt 2}\sum_{\nu (\mathrm{mod} \ 4)}f|_k B\widetilde{A}_\nu,$$
where $$B=\lt(\mat 4&1\\ 0&4 \emat, e^{(2\lambda+1)\pi i/4}\rt)\quad \mathrm{and}\quad A_\nu= \mat1&0\\4\nu&1\emat.$$
For each $m$ satisfying $(-1)^\lambda m\equiv 0,1 \pmod 4$ and $Re(s)>1$,
we define the Poincar\'{e} series $F_{m,k,N}^+(z,s)$ by
\begin{equation}F_{m,k,N}^+(z,s)=pr_k^+(F_{m,k,N}(z,s)).\end{equation}
The function $F_{m,k,N}^+(z,s)$ has weight $k$ for $\G_0(N)$ and satisfies
\begin{equation}\label{hyppp}
\Delta_k F_{m,k,N}^+(z,s)=\lt(s-\frac{k}{2}\rt)\lt(1-\frac k2-s\rt)F_{m,k,N}^+(z,s)\end{equation}
as $F_{m,k,N}(z,s)$ does. The Fourier expansion of the weakly harmonic Maass form $F_{m,k,4}^+(z,k/2)$ is computed in \cite[Theorem 3.5]{BJO} when $m<0$ and $k\geq 3/2.$  In fact, $F_{m,k,4}^+(z,k/2)$ is a weakly holomorphic modular form if $k>3/2$ and $m<0$. The Fourier expansion of $F_{m,k,4}^+(z,1-k/2)$ for $k\leq 1/2$ and $m<0$ is calculated in \cite[Theorem 2.1]{BO1}. Now we establish the Fourier coefficients of $F_{m,k,N}^+(z,s)$ for any $m$, $k$, $N$ and $s$ by proving the following three lemmas first.

\begin{lem}\cite[Proposition 3]{Kohnen}\label{Kpr} For $\alpha=\left(\frac{-4}{N'}\right)$, we set
$$\eta^{\left(\frac{-\alpha}{N'}\right)}=\left(\mat 1&0\\-\alpha N'&1\emat,(-\alpha N'z+1)^k\right)\ \mathrm{and} \ \eta^{\left(\frac{1}{2N'}\right)}=\left(\mat 1&0\\ 2N'&1\emat,(2N'z+1)^k\right).$$
For a weight $k$ weakly harmonic Maass form $g$ on $\G_0(N)$, if we write
\begin{eqnarray*}
&&g(z)=\sum_n a(n,y)q^n\\
&&g|_k\eta^{\left(\frac{-4}{N'}\right)}(z)=\sum_n a^{\left(\frac{-\alpha}{N'}\right)}(n,y)q^{n/4}\\
&&g|_k\eta^{\left(\frac{1}{2N'}\right)}(z)=\sum_{(-1)^\lambda n\equiv 1(4)}a^{\left(\frac{1}{2N'}\right)}(n,y)q^{n/4},
\end{eqnarray*}
then
\begin{eqnarray*} pr_k^+(g)(z)&=&\sum_{n\equiv 0\,(4)}
\left(a(n,y)+(1-(-1)^\lambda i)2^{2\lambda-1}i^{n/4}a^{(-\alpha/N')}
\left(\frac{n}{4},16y\right)\right)q^n\\
& &+\sum_{(-1)^\lambda n\equiv 1\,(4)}
\left(a(n,y)+2^{\lambda-1}\left(\frac{(-1)^\lambda n}{2}\right)a^{(1/N')}
(n,4y)\right)q^n.
\end{eqnarray*}
\end{lem}

The original \cite[Proposition 3]{Kohnen} is for a cusp form $g$, but the result still holds for a weakly harmonic Maass form $g$.  As Fourier coefficients of $F:=F_{m,k,N}(z,s)$ are already computed in Theorem \ref{exp1}, we compute those for $F|_k\eta^{\left(\frac{-4}{N'}\right)}(z,s)$ and for  $F|_k\eta^{\left(\frac{1}{2N'}\right)}(z,s)$ in two lemmas below to find the Fourier expansion of $F_{m,k,N}^+(z,s)$.

\begin{lem}\label{exp2}
If $m$ is an integer, then $F|_k\eta^{\left(\frac{-4}{N'}\right)}(z,s)$ has the Fourier expansion
$$F_{m,k,N}|_k\eta^{\left(\frac{-4}{N'}\right)}(z,s)=\sum_{n\in \mathbb{Z}} c'_{m,k}(n,s)\mathcal{W}_{n,k}\left(\frac{y}{4},s\right) e\left(\frac{nx}{4}\right),$$
where the coefficients $c'_{m,k}(n,s)$ are given by
$$2\pi i^{-k}\frac{(-1)^\lambda i}{4}i^{-n}\sum_{c>0 \atop c\equiv 1 (2)} H\lt(\frac{m}{4},n\rt)\times \left\{
    \begin{array}{ll}
      \displaystyle{|4mn|^{\frac{1-k}{2}}J_{2s-1}\lt(\frac{4\pi\sqrt{|4mn|}}{Nc}\rt)}, & mn>0, \\
      \displaystyle{|4mn|^{\frac{1-k}{2}}I_{2s-1}\lt(\frac{4\pi\sqrt{|4mn|}}{Nc}\rt)}, & mn<0, \\
    \displaystyle{2^{k-1}\pi^{s+\frac{k}{2}-1}|4(m+n)|^{s-\frac{k}{2}}(Nc)^{1-2s}}, & mn=0, m+n\neq 0,
\\
    \displaystyle{2^{2k-2}\pi^{k-1}\Gamma(2s)(2Nc)^{1-2s}4^{s-k/2}}, & m=n=0, \end{array}
  \right.
$$
where $H(\frac{m}{4},n)=\lt(\frac{4}{-N'c}\rt)\lt(\frac{-4}{N'c}\rt)^{-k}\sum_{\delta(N'c)^*}\lt(\frac{\delta}{N'c}\rt)
e\lt(\frac{n\delta+m\delta^{-1}}{N'c}\rt)$.
\end{lem}
\begin{proof} By the exactly same argument in the proof of \cite[Proposition 4]{Kohnen} again, we get
\begin{eqnarray*}
&&F|_k\eta^{\left(\frac{-4}{N'}\right)}(z,s)=\sum_{c>0,c\equiv1\,(2)}\lt(\frac{-4}{-N'c}\rt)^k\sum_{d(Nc)^*,\, d\equiv -N'c(4)}
\lt(\frac{Nc}{d}\rt)\\
    &&\qquad\qquad\times\sum_{r\in\Z}\lt(Nc\lt(\frac{z}{4}+r\rt)+d\rt)^{-k}
\varphi_{m,k}\lt(4\frac{a\lt(\frac{z}{4}+r\rt)+\frac{b}{4}}{Nc\lt(\frac{z}{4}+r\rt)+d},s\rt).
\end{eqnarray*}
The Fourier expansion of the inner sum is given by Lemma \ref{gendit} and
we see that
\begin{eqnarray*}
&&\frac{1}{Nc}\lt(\frac{-4}{-N'c}\rt)^k\sum_{d(Nc)^*,\, d\equiv -N'c(4)}
\lt(\frac{Nc}{d}\rt)e\lt(\frac{nd+4md^{-1}}{Nc}\rt)\\
&&\quad=\frac{(-1)^\lambda i}{4}i^{-n}\lt(\frac{4}{-N'c}\rt)\lt(\frac{-4}{N'c}\rt)^{-k}\sum_{d(N'c)^*}\lt(\frac{d}{N'c}\rt)
e\lt(\frac{nd+4^{-1}md^{-1}}{N'c}\rt)\\
&&\quad=\frac{(-1)^\lambda i}{4}i^{-n}H(m/4,n),
\end{eqnarray*}
from which Lemma \ref{exp2} follows.

\end{proof}

\begin{lem}\label{exp3}
If $m$ is an integer, then $F|_k\eta^{\left(\frac{1}{2N'}\right)}(z,s)$ has the Fourier expansion
$$F_{m,k,N}|_k\eta^{\left(\frac{1}{2N'}\right)}(z,s)=\sum_{n\in \mathbb{Z}} c''_{m,k}(n,s)\mathcal{W}_{n,k}\left(\frac{y}{4},s\right) e\left(\frac{nx}{4}\right),$$
where the coefficients $c''_{m,k}(n,s)$ are given by
$$2\pi i^{-k}\sum_{c>0 \atop c\equiv 1 (2)} \frac{K_k(4m,n,2Nc)}{2Nc}\times \left\{
    \begin{array}{ll}
      \displaystyle{|4mn|^{\frac{1-k}{2}}J_{2s-1}\lt(\frac{4\pi\sqrt{|4mn|}}{2Nc}\rt)}, & mn>0, \\
      \displaystyle{|4mn|^{\frac{1-k}{2}}I_{2s-1}\lt(\frac{4\pi\sqrt{|4mn|}}{2Nc}\rt)}, & mn<0, \\
    \displaystyle{2^{k-1}\pi^{s+\frac{k}{2}-1}|4(m+n)|^{s-\frac{k}{2}}(2Nc)^{1-2s}}, & mn=0, m+n\neq 0,
\\
    \displaystyle{2^{2k-2}\pi^{k-1}\Gamma(2s)(4Nc)^{1-2s}4^{s-k/2}}, & m=n=0.    \end{array}
  \right.
$$
\end{lem}
\begin{proof} Arguing analogously as in the proof of Lemma \ref{exp2}, we obtain
\begin{eqnarray*}
F|_k\eta^{\left(\frac{1}{2N'}\right)}(z,s)&=&\sum_{c>0,c\equiv1\,(2)}\sum_{d(2Nc)^*}
\lt(\frac{2Nc}{d}\rt)\lt(\frac{-4}{d}\rt)^k\\
&&\times\sum_{r\in\Z}\lt(2Nc\lt(\frac{z}{4}+r\rt)+d\rt)^{-k}
\varphi_{m,k}\lt(4\frac{a\lt(\frac{z}{4}+r\rt)+\frac{b}{4}}{2Nc\lt(\frac{z}{4}+r\rt)+d},s\rt).
\end{eqnarray*}
Now Lemma \ref{exp3} follows from Lemma \ref{gendit}.
\end{proof}

Applying Lemma \ref{Kpr} together with Theorem \ref{exp1}, Lemma \ref{exp2} and Lemma \ref{exp3}, we have

\begin{thm}\label{fourierfp}For any $m$ and $s$ satisfying $(-1)^\lambda m\equiv 0,1 \pmod 4$ and $Re(s)>1$, $F_{m,k,N}^+(z,s)$ has the Fourier expansion
$$F_{m,k,N}^+(z,s)=\mathcal{M}_{m,k}(y,s)e(mx)+\sum_{(-1)^\lambda n\equiv 0,1(4)} b_{m,k}(n,s)\mathcal{W}_{n,k}(y,s)e(nx),$$
where the coefficients $b_{m,k}(n,s)$ are given by
$$2\pi i^{-k}\sum_{c>0} \lt(1+\lt(\frac{4}{N'c}\rt)\rt)\frac{K_k(m,n,Nc)}{Nc}\times \left\{
    \begin{array}{ll}
      \displaystyle{|mn|^{\frac{1-k}{2}}J_{2s-1}\lt(\frac{4\pi\sqrt{|mn|}}{Nc}\rt)}, & mn>0, \\
      \displaystyle{|mn|^{\frac{1-k}{2}}I_{2s-1}\lt(\frac{4\pi\sqrt{|mn|}}{Nc}\rt)}, & mn<0, \\
    \displaystyle{2^{k-1}\pi^{s+\frac{k}{2}-1}|m+n|^{s-\frac{k}{2}}(Nc)^{1-2s}}, & mn=0, m+n\neq 0,
\\
    \displaystyle{2^{2k-2}\pi^{k-1}\Gamma(2s)(2Nc)^{1-2s}}, & m=n=0.        \end{array}
  \right.
$$
\end{thm}
\begin{proof}
If $n\equiv 0\,\,({\rm mod}\,\,4)$, then we find that
$$\lt(1+\lt(\frac{4}{N'c}\rt)\rt)K_k(m,n,Nc)=\left\{
    \begin{array}{ll}
      \displaystyle{K_k(m,n,Nc)}, & c\equiv 0(2), \\
\\
      \displaystyle{K_k(m,n,Nc)+(1-(-1)^\lambda i)\frac{(-1)^\lambda i}{4}H\lt(\frac{m}{4},\frac{n}{4}\rt)}, & c\equiv 1(2).
    \end{array}
  \right.
$$
But if $(-1)^\lambda n\equiv 1\,\,({\rm mod}\,\,4)$, then we have
$$\lt(1+\lt(\frac{4}{N'c}\rt)\rt)K_k(m,n,Nc)=\left\{
    \begin{array}{ll}
      \displaystyle{K_k(m,n,Nc)}, & c\equiv 0(2), \\
\\
      \displaystyle{K_k(m,n,Nc)+\frac{1}{\sqrt{2}}\frac{(-1)^\lambda i}{4}K_k(4m,n,2Nc)}, & c\equiv 1(2).
    \end{array}
  \right.
$$
Theorem now follows from  Lemma \ref{Kpr}, Theorem \ref{exp1}, Lemma \ref{exp2} and Lemma \ref{exp3} altogether.
\end{proof}

\section{Weight 3/2 weakly harmonic Maass forms}

 By the theory of the resolvent kernel again, Poincar\'{e} series $F^+_{m,k,N}(z,s)$ has an analytic continuation in $s$ to  $\mathrm{Re}(s)>1/2$ except for possibly finitely many simple poles in $(1/2,1)$. In fact,  for $d\equiv 0,1 \pmod 4$ and $D\equiv 0,3 \pmod 4$,
\begin{equation}\label{for12}F_{d,1/2,4}^+(z,s)=\M_{d,1/2}(y,s)e(dx)+\sum_{-D\equiv 0,1\ (4)}b_{d,1/2}(-D,s)\W_{-D,1/2}(y,s)e(-Dx)\end{equation}
has a simple pole at $s=3/4$. Hence Duke, Imamo\={g}lu and T\'{o}th \cite{DIT} used a weak Maass form
 \begin{equation}\label{h}h_{d,1/2}(z,s):=F_{d,1/2,4}^+(z,s)-\frac{b_{d,1/2}(0,s)}{b_{0,1/2}(0,s)}F_{0,1/2,4}^+(z,s)\quad (d\neq 0)\end{equation} that is holomorphic at $s=3/4$ and set $h_{d,1/2}(z):=h_{d,1/2}(z,3/4)$ if $d\neq0$ and $h_{0,1/2}(z):=f_0(z)$ so that $h_{d,1/2}(z)$ form a basis for $H_{1/2}^!$. Furthermore, they proved that $h_{d,1/2}(z)=f_d(z)$ if $d<0$ and $\xi_{1/2}(h_{d,1/2}(z))=-2\sqrt d g_{-d}(z)\in M_{3/2}^!$  if $d>0$ \cite[Proposition 1]{DIT}.

It requires a slightly different approach to construct a basis for $H_{3/2}^!$. Let $-D\equiv 0,1 \pmod 4$ and $F_D^+(z,s):=F_{D,3/2,4}^+(z,s)$. Then by Theorem \ref{fourierfp}, the Fourier expansion of $F_D^+(z,s)$ is given by
\begin{equation}\label{for32}F_{D}^+(z,s)=\M_{D,3/2}(y,s)e(Dx)+\sum_{-d\equiv 0,3\ (4)}b_{D,3/2}(-d,s)\W_{-d,3/2}(y,s)e(-dx).\end{equation}
First, note from (\ref{mk2}) that
\begin{equation}\label{m3w}\lim_{s\to 3/4}\M_{D,3/2}(y,s)e(Dx)=\left\{
    \begin{array}{ll}
    \frac{2}{\sqrt \pi}q^D,& D\neq 0,\\
    1, & D=0.\end{array}
  \right.\end{equation}
Also, it follows from (\ref{defmnwn}), (\ref{wk2}), and (\ref{wi}) that
\begin{equation}\label{wn32}
\lim_{s\to 3/4}\W_{-d,3/2}(y,s)=\left\{
    \begin{array}{ll}
    \displaystyle{\frac{2}{\sqrt{\pi}}|d|^{\frac{1}{2}}e^{2\pi dy}}& d<0, \\
    \displaystyle{\lim_{s\to 3/4}\G(s-\frac 34)^{-1}d^{1/2}e^{2\pi dy}\G(-\frac 12,4\pi dy)}, & d>0, \\
     \displaystyle{\lim_{s\to 3/4}\G(s-\frac 34)^{-1}\frac{2}{\pi \sqrt y}}, & d=0.
     \end{array}
  \right.
\end{equation}
Now recall the symmetric property of the Kloosterman sum (\cite[Proposition 3.1]{BO1})
$$K_{\frac{3}{2}}(m,n,c)=K_{\frac{3}{2}}(n,m,c)=-iK_{\frac{1}{2}}(-m,-n,c),$$
which implies

\begin{equation}\label{sym}b_{D,3/2}(-d,s)=b_{-D,1/2}(d,s)\times\left\{
    \begin{array}{ll}
      \displaystyle{-|Dd|^{-\frac{1}{2}}}, & Dd\neq 0, \\
    \displaystyle{-2\pi^{\frac 12}|D-d|^{-\frac 12}}, & Dd=0, D-d\neq 0,
\\
    \displaystyle{-4\pi}, & D=d=0.        \end{array}
  \right.
  \end{equation}
It is proved in \cite[Lemma 3, (2.25), (2.26)]{DIT} that $b_{-D,1/2}(d,s)$ has a simple pole at $s=\frac{3}{4}$ if and only if both $d$ and $-D$ are equal to zero or non-zero squares and
\begin{equation}\label{residue12}
 \textrm{Res}_{s=3/4}b_{-D,1/2}(d,s)=\left\{\begin{array}{ll}
{\frac{3}{16\pi}},  &D=0, d=0,\\
&\\
  {\frac{3m}{2\pi}}, &D=0, d=m^2\neq 0,\\
  &\\
   {\frac{12m\sqrt{-D}}{\pi}}, &-D=\square\neq 0, d=m^2\neq 0.
  \end{array}
  \right.
\end{equation}
Accordingly, the coefficient $b_{D,3/2}(-d,s)$ has
a simple pole at $s=\frac{3}{4}$ if and only if both $-D$ and $d$ are equal to zero or non-zero squares as well, but in which case the singularity cancels with the zero of $\Gamma(s-3/4)^{-1}$ appearing in $\mathcal{W}_{-d,3/2}(y,s)$. Therefore $F_D^+(z,s)$ is holomorphic in $s$.   More precisely, if $-D$ is a non-zero square, then it follows from (\ref{wn32}), (\ref{sym}) and (\ref{residue12}) that
\begin{eqnarray}\label{bw32}
&&\lim_{s\to 3/4}b_{D,3/2}(-d,s)\W_{-d,3/2}(y,s)\cr
&&\qquad=\left\{
    \begin{array}{ll}
        \displaystyle{(\mathrm{Res}_{s=3/4}b_{D,3/2}(-m^2,s))|m|e^{2\pi m^2y}\G(-\frac 12,4\pi m^2y)}, & d=m^2\neq 0, \\
        \displaystyle{(\mathrm{Res}_{s=3/4}b_{D,3/2}(0,s))\frac{2}{\pi \sqrt y}}, & d=0,
     \end{array}
  \right.\cr
&&\qquad=\left\{
    \begin{array}{ll}
       \displaystyle{-\mathrm{Res}_{s=3/4}b_{-D,1/2}(m^2,s)|D|^{-1/2}e^{2\pi m^2y}\G(-\frac 12,4\pi m^2y)}, & d=m^2\neq 0, \\
     \displaystyle{-\mathrm{Res}_{s=3/4}b_{-D,1/2}(0,s)\frac{4}{\sqrt{\pi|D|y}}}, & d=0,
     \end{array}
  \right. \cr
&&\qquad=\left\{
    \begin{array}{ll}
       \displaystyle{-\frac{12m}{\pi}e^{2\pi m^2y}\G(-\frac 12,4\pi m^2y)}, & d=m^2\neq 0, \\
     \displaystyle{-\pi^{-3/2}\frac{6}{\sqrt{y}}}, & d=0,
     \end{array}
  \right.
\end{eqnarray}
where the last is obtained from the fact $b_{0,1/2}(d,s)=b_{d,1/2}(0,s)$.
Similarly, if $D=0$, then we have
\begin{eqnarray}\label{bw032}
\lim_{s\to 3/4}b_{0,3/2}(-d,s)\W_{-d,3/2}(y,s)
   &=&\left\{
    \begin{array}{ll}
       \displaystyle{-\frac{3m}{\sqrt\pi}e^{2\pi m^2y}\G(-\frac 12,4\pi m^2y)}, & d=m^2\neq 0, \\
     \displaystyle{-\frac{3}{2\pi\sqrt{y}}}, & d=0.
     \end{array}
  \right.
\end{eqnarray}
Thus we reach the following result.
\begin{prop}\label{prop1}
Let $D\equiv 0,3 \pmod 4$.

(1) If $D$ is positive, then $\displaystyle{F_D^+(z,3/4)=0.}$

(2) If $D$ is negative and $-D$ is not a square, then
$$F_{D}^+(z,3/4)\in M_{3/2}^!.$$

(3) If $D$ is negative and $-D$ is a non-zero square, then
$$F_{D}^+(z,3/4)-\frac{4}{\sqrt\pi}F_{0}^+(z,3/4)\in M_{3/2}^!.$$
\end{prop}
\begin{proof} Suppose $D>0$ or $D<0$ but $-D\neq \square$. Then there is no pole at $s=3/4$ and
\begin{equation}\label{bw32p}
b_{D,3/2}(-d,3/4)\W_{-d,3/2}(y,3/4)e(-dx)=\left\{
    \begin{array}{ll}
    \displaystyle{\frac{2}{\sqrt{\pi}}b_{D,3/2}(-d,3/4)|d|^{1/2}q^{-d}}, & d<0; \\
    0, & d\geq 0
     \end{array}
  \right.
\end{equation} which yields \begin{equation*}\label{whm32n}F_{D}^+(z,3/4)=\frac{2}{\sqrt \pi}\lt(q^D+\sum_{0<-d\equiv 0,3\ (4)}b_{D,3/2}(-d,3/4)|d|^{1/2}q^{-d}\rt).\end{equation*} This proves (2) and (1) as well, because there is no weight 3/2 cusp form satisfying the plus space condition.
For $-D=\square$, it follows from (\ref{m3w}), (\ref{wn32}), (\ref{bw32}), (\ref{bw032}) and (\ref{bw32p}) that
\begin{eqnarray}\label{for32full}F_{D}^+(z,3/4)= \frac{2}{\sqrt \pi}q^D-\pi^{-3/2}\frac{6}{\sqrt{y}}+\sum_{0<-d\equiv 0,3\ (4)}\frac{2}{\sqrt \pi}b_{D,3/2}(-d,3/4)|d|^{1/2}q^{-d}\cr
-\frac{12}{\pi} \sum_{0<d= \square}\sqrt d \G(-\frac 12,4\pi |d|y)q^{-d}\quad \textrm{if}\quad 0\neq -D=\square\end{eqnarray}
and
\begin{eqnarray}\label{eisen32}F_{0}^+(z,3/4)= 1-\frac{3}{2\pi\sqrt y}+\sum_{0<-d\equiv 0,3\ (4)}\frac{2}{\sqrt \pi}b_{0,3/2}(-d,3/4)|d|^{1/2}q^{-d}\cr
-\frac{3}{{\sqrt \pi}}\sum_{0<d= \square} \sqrt d\G(-\frac 12,4\pi |d|y)q^{-d},\quad \textrm{if}\quad D=0,\end{eqnarray} which proves (3). Alternatively, (3) can be proved by showing $\xi_{3/2}(F_{D}^+(z,3/4)-\frac{4}{\sqrt\pi}F_{0}^+(z,3/4))=0.$
\end{proof}
The Poincar\'{e} series $F_{0}^+(z,3/4)=-12E(z)$, where $E(z)$ is Zagier's Eisenstein series of weight $3/2$ for $\G_0(4)$ (\cite{Z1}) and
$$E(z)=\sum_{n=0}^\i H(n)q^n+\frac{1}{16\pi\sqrt y}\sum_{n=-\i}^\i\beta(4\pi n^2 y)q^{-n^2}.$$ Here $H(n)$ is the Hurwitz class number with $H(0)=\zeta(-1)=-\frac{1}{12}$ and $\beta(s)=\int_1^\i t^{-3/2}e^{-st}dt$. This yields for positive integer $n$,
\begin{equation}\label{cn}H(n)=-\frac{1}{6}\sqrt{\frac{n}{\pi}}b_{0,3/2}(n,3/4).\end{equation}
The results (2) and (3) of Proposition \ref{prop1} were proved in \cite[Proposition 3.6]{BJO} in a slightly different form. The series in (3) reflects a symmetric relation with the weight $1/2$ weakly harmonic Maass form in (\ref{h}) as $$F_{D}^+(z,3/4)-\frac{4}{\sqrt\pi}F_{0}^+(z,3/4)=
\lim_{s\to 3/4}\lt(F_D^+(z,s)-\frac{b_{D,3/2}(0,s)}{b_{0,3/2}(0,s)}F_0^+(z,s)\right).$$

In order to construct a nontrivial weakly harmonic Maass form for the case $D>0$, we consider the derivative of the Poincar\'{e} series $F_{D}^+(z,s)$ at $s=3/4$.
\begin{prop}\label{prop2} Let $0<D\equiv 0,3 \pmod 4$. Then
 \begin{align}\label{kdpf}
\pds F_D^+(z,s)|_{s=3/4}=&\displaystyle -12\frac{H(D)}{\sqrt{\pi D y}}-i\G(-\frac 12,-4\pi Dy)q^D-2\sqrt\pi i q^D\\
&+\displaystyle\sum_{{0<-d\equiv 0,3\ (4)}\atop{-d\neq D}}\frac{\partial}{\partial s}b_{D,3/2}(-d,s)|_{s=\frac{3}{4}}\frac{2\sqrt{|d|}}{\sqrt\pi}q^{-d}\cr
&+\displaystyle\sum_{0<d\equiv 0,1\ (4)}b_{D,3/2}(-d,3/4)\sqrt{|d|}\Gamma\left(-\frac{1}{2},4\pi dy\right)q^{-d}.\nonumber
 \end{align}
\end{prop}
\begin{proof}
Applying (\ref{for32}), (\ref{wn32}) and Proposition \ref{prop1}(1), we have
\begin{align*}
\frac{\partial}{\partial s}F_D^+(z,s)|_{s=\frac{3}{4}}
=& \mathcal X(y)e(Dx) +\displaystyle\sum_{{n>0}\atop{n\neq D}}\frac{\partial}{\partial s}b_{D,3/2}(n,s)|_{s=\frac{3}{4}}\frac{2\sqrt{n}}{\sqrt\pi}q^n\\
&+\displaystyle\sum_{n\leq 0}b_{D,3/2}(n,3/4)\pds\mathcal W_{n,3/2}(y,s)|_{s=\frac{3}{4}}e(nx).
\end{align*}
where $\mathcal X(y)=\displaystyle\frac{\partial}{\partial s}\left(\mathcal M_{D,3/2}(y,s)+b_{D,3/2}(D,s)\mathcal W_{D,3/2}(y,s)\right)|_{s=\frac{3}{4}}$.
It follows from (\ref{defmnwn}), (\ref{wi}) and the fact $W_{\mu,\nu}=W_{\mu,-\nu}$ that
$\pds\mathcal W_{n,3/2}(y,s)|_{s=\frac{3}{4}}=\sqrt{|n|}e^{2\pi|n|y}\Gamma\left(-\frac{1}{2},4\pi|n|y\right)$ if $n<0$. Also, $\pds\mathcal W_{0,3/2}(y,s)|_{s=\frac{3}{4}}=\frac{2}{\pi\sqrt y}.$ Hence
\begin{align}\label{kdp}
\frac{\partial}{\partial s}F_D^+(z,s)|_{s=\frac{3}{4}}
=& \mathcal X(y)e(Dx) +\displaystyle\sum_{{n>0}\atop{n\neq D}}\frac{\partial}{\partial s}b_{D,3/2}(n,s)|_{s=\frac{3}{4}}\frac{2\sqrt{n}}{\sqrt\pi}q^n\cr
&+\displaystyle b_{D,3/2}(0,3/4)\frac{2}{\pi\sqrt y}+\displaystyle\sum_{n< 0}b_{D,3/2}(n,3/4)\sqrt{|n|}\Gamma\left(-\frac{1}{2},4\pi|n|y\right)q^n.
\end{align}
Note from (\ref{sym}) and the fact $h_{d,1/2}=f_d$ for $d<0$ that $b_{D,3/2}(D,3/4)=-\frac1 D b_{-D,1/2}(-D,3/4)=-\frac{1}{\sqrt D}$ and $M_{\frac 34,\frac14}(y)=W_{\frac 34,\frac14}(y)=W_{\frac 34,-\frac14}(y)=y^{3/4}e^{-y/2}$. We then have
\begin{align*}\mathcal X(y)&=\left[\pds\left(\G(2s)^{-1}(4\pi Dy)^{-3/4}M_{\frac 34,s-1/2}(4\pi Dy)-\G(s+\frac34)^{-1}(4\pi Dy)^{-3/4}W_{\frac 34,s-1/2}(4\pi Dy)\right)\right]_{s=\frac 34}\\
&=\mathcal A+\frac{2}{\sqrt \pi}(4\pi Dy)^{-3/4}\mathcal B,\end{align*}
where
$$\mathcal A=\frac{d}{ds}\left(\G(2s)^{-1}-\G(s+\frac34)^{-1}\right)(4\pi Dy)^{-3/4}W_{\frac 34,\frac14}(4\pi Dy)=-\frac{2}{\sqrt \pi}(2-\gamma-2\ln2)e^{-2\pi Dy}$$ for the Euler constant $\g$ and
\begin{align*}\mathcal B&=\left[\pds\left( M_{\frac 34,s-1/2}(4\pi Dy)- W_{\frac 34,s-1/2}(4\pi Dy)\right)\right]_{s=\frac 34}\\
&=\left[\pds\left( M_{\frac 34,s-1/2}(4\pi Dy)-\frac{\G(1-2s)}{\G(1/4-s)} M_{\frac 34,s-1/2}(4\pi Dy)-\frac{\G(2s-1)}{\G(s-3/4)} M_{\frac 34,1/2-s}(4\pi Dy)\right)\right]_{s=\frac 34}\\
&=\left[-\frac{d}{ds}\frac{\G(1-2s)}{\G(1/4-s)}\right]_{s=\frac 34} M_{\frac 34,\frac14}(4\pi Dy)-\left[\frac{d}{ds}\frac{\G(2s-1)}{\G(s-3/4)}\right]_{s=\frac 34} M_{\frac 34,-\frac 14}(4\pi Dy)\\
&=(2-\gamma-2\ln2)M_{\frac 34,\frac14}(4\pi Dy)-\sqrt \pi M_{\frac 34,-\frac 14}(4\pi Dy)\\
&=(2-\gamma-2\ln2)(4\pi Dy)^{3/4}e^{-2\pi Dy}-\sqrt \pi M_{\frac 34,-\frac 14}(4\pi Dy).
\end{align*}
The second equality above holds by (\ref{mwrel}). But by (\ref{Kummer}), (\ref{hypM}), (\ref{hypU}) and the transformation $M(a,b,z)=e^zM(b-a,b,-z)$,
\begin{align*}
M_{\frac 34,-\frac 14}(4\pi Dy)&=e^{-2\pi Dy}(4\pi Dy)^{1/4}M(-\frac 12,\frac 12,4\pi Dy)\\
&=e^{2\pi Dy}(4\pi Dy)^{1/4}M(1,\frac 12,-4\pi Dy)\\
&=e^{2\pi Dy}(4\pi Dy)^{1/4}\left(\frac 12 U(1,\frac 12,-4\pi Dy)+\sqrt\pi(-4\pi Dy)^{1/2} M(\frac 32,\frac 32,-4\pi Dy)\right)\\
&=\frac 12e^{2\pi Dy}(4\pi Dy)^{1/4} U(1,\frac 12,-4\pi Dy)+\sqrt\pi i (4\pi Dy)^{3/4} e^{-2\pi Dy}\\
&=\frac i2e^{-2\pi Dy}(4\pi Dy)^{3/4} \G(-\frac 12,-4\pi Dy)+\sqrt\pi i (4\pi Dy)^{3/4} e^{-2\pi Dy},
\end{align*}
where the last line is obtained from $\G(s,z)=e^{-z}z^sU(1,1+s,z)$. Hence,
$$\mathcal X(y)e(Dx)=-i\G(-\frac 12,-4\pi Dy)q^D-2\sqrt\pi i q^D,$$
and finally we establish (\ref{kdpf}) by applying  (\ref{sym}) and (\ref{cn}).
 \end{proof}

In general, a harmonic weak Maass form has a Fourier expansion
$$h(z)=\sum_{n\gg-\i}c_h^+(n)q^n+c_h^-(0)y^{1-k}+\sum_{0\neq n\ll \i}c_h^-(n)\G(1-k,-4\pi ny)q^n$$ so that
 \begin{equation}\label{sh}\xi_k(h)=(1-k)\overline{c_h^-(0)}-\sum_{0\neq n\ll \i}\overline{c_h^-(n)(-4\pi n)^{1-k}}q^{-n}.\end{equation}
This implies with (\ref{kdpf}) that
\begin{align}\label{shpdsf}
\xi_{3/2}\left(\frac{\partial}{\partial s}F_D^+(z,s)|_{s=\frac{3}{4}}\right)
&= {\frac{6H(D)}{\sqrt{\pi D}}}+\frac{1}{\sqrt{4\pi D}}q^{-D} -\frac{1}{\sqrt{4\pi}}\displaystyle\sum_{0<d\equiv 0,1\ (4)}\overline{b_{D,3/2}(-d,3/4)}q^d\cr
&={\frac{6H(D)}{\sqrt{\pi D}}}+\frac{1}{\sqrt{4\pi D}}\left(q^{-D} +\displaystyle\sum_{0<d\equiv 0,1\ (4)}\overline{b_{-D,1/2}(d,3/4)}\frac{1}{\sqrt{|d|}}q^d\right).
\end{align}
Using (\ref{h}), (\ref{sym}) and (\ref{residue12}), we may write the weight $1/2$ weakly harmonic Maass form
$$h_{d,1/2}(z)=\M_{d,1/2}(y,3/4)e(dx)-\delta_{\square,d}8\sqrt{dy}+\sum_{0\neq n\equiv 0,1\ (4)}a_d(n)\W_{n,1/2}(y,3/4)e(nx),$$ where
$$a_d(n)=b_{d,1/2}(n,3/4)-\delta_{\square,n}\frac{4\sqrt{nd}}{\sqrt\pi}b_{-d,3/2}(0,3/4).$$
Since $h_{d,1/2}(z)=f_d(z)$, which is given in (\ref{fn}) if $d<0$ (\cite[Proposition 1]{DIT}), we can observe using (\ref{cn}) that
\begin{align}\label{shf}
\xi_{3/2}\left(\frac{\partial}{\partial s}F_D^+(z,s)|_{s=\frac{3}{4}}\right)=\frac{1}{\sqrt{4\pi D}}f_{-D}(z)+{\frac{6H(D)}{\sqrt{\pi D}}}\theta(z),\end{align}
where $\theta(z)=\sum_{n\in\Z}q^{n^2}.$ We are now ready to construct a basis for the space of weight $3/2$ weakly harmonic Maass forms.

\begin{thm}\label{main1} For each $D\equiv 0,3 \pmod 4$, let us define
\begin{equation*}h_{D,3/2}(z):=\left\{
    \begin{array}{ll}
      \displaystyle{\pds F_D^+(z,s)|_{s=3/4}-8\sqrt{\frac{\pi}{D}}H(D)F_0^+(z,3/4)}, & D>0, \\
      \displaystyle{F_{D}^+(z,3/4)-\frac{4}{\sqrt\pi}F_{0}^+(z,3/4)}, & -D=\square,\\
    \displaystyle{F_D^+(z,3/4)}, & D=0\quad \textrm{or}\quad 0<-D\neq \square.
 \end{array}
  \right.\end{equation*}
 The set $\{h_{D,3/2}|D\equiv 0,3 \pmod 4 \}$ forms a basis for $H_{3/2}^!$. When $D<0$, $h_{D,3/2}(z)$  forms a unique basis for $M_{3/2}^!$ of the form given in (\ref{gm}), because
 $$h_{D,3/2}(z)=\frac{2}{\sqrt \pi}g_D(z).$$
When $D\geq 0$,  $h_{D,3/2}(z)$ is a unique weakly harmonic Maass form with bounded holomorphic part satisfying
 \begin{equation}\label{shadowf} \xi_{3/2}\left(h_{D,3/2}(z)\right)=\frac{1}{2\sqrt{\pi D}}f_{-D}(z)\ (D>0)\quad \textrm{and}\quad\xi_{3/2}\left(h_{0,3/2}(z)\right)=\frac{3}{4{\pi}}f_{0}(z),\end{equation}
 where $f_d$ is given by (\ref{fn}).
\end{thm}
\begin{proof}
As it satisfies
\begin{equation}\label{hypppd}
\Delta_{3/2} \pds F_{D}^+(z,s)=\lt(s-\frac{3}{4}\rt)\lt(\frac 14-s\rt)\pds F_{D}^+(z,s),\end{equation} $h_{D,3/2}(z)$ is clearly a weakly harmonic Maass form for each $D$.  Furthermore, Proposition \ref{prop1} and its proof imply that $h_{D,3/2}(z)=\frac{2}{\sqrt \pi}g_D \in M_{3/2}^!$ when $D<0$ and the weight $3/2$ non-holomorphic Eisenstein series $h_{0,3/2}(z)=-12E(z)$ is a well known weakly harmonic Maass form having
\begin{equation}\label{shzero}\xi_{3/2}(h_{0,3/2}(z))=-12(-\frac{1}{16\pi})\theta(z)=\frac{3}{4\pi}f_0(z).\end{equation}
For $D>0$, it follows from the definition of $h_{D,3/2}(z)$, (\ref{sh}), (\ref{shf}) and (\ref{shzero}) that $$\xi_{3/2}\left(h_{D,3/2}(z)\right)=\frac{1}{2\sqrt{\pi D}}f_{-D}(z).$$
Finally, the uniqueness is a consequence of the fact that the only holomorphic modular form of weight $3/2$ satisfying the plus space condition is zero.
\end{proof}

\begin{proof}[Proof of Theorem \ref{mainmock}] We take the holomorphic part of $2\sqrt{\pi D}h_{D,3/2}(z)$ as $g_D(z)$ if $D>0$ and the holomorphic part of $\displaystyle{\frac{4\pi}{3}h_{0,3/2}(z)}$  as $g_0(z)$. Then Theorem \ref{mainmock} follows immediately from Theorem \ref{main1}.
\end{proof}

\end{document}